\documentclass[11pt]{amsart}

\usepackage{lmodern} 
\usepackage{microtype} 
\usepackage[english]{babel} 
\usepackage{mathrsfs} 
\usepackage{hyperref} 
\usepackage[hyperpageref]{backref} 
\usepackage[msc-links]{amsrefs}

\theoremstyle{plain} 
\newtheorem{theorem}{Theorem} 
\newtheorem{lemma}[theorem]{Lemma} 
\newtheorem{corollary}[theorem]{Corollary}

\theoremstyle{definition} 
\newtheorem*{definition}{Definition}

\theoremstyle{remark} 
\newtheorem*{remark}{Remark}

\DeclareMathOperator{\mre}{Re} 
\DeclareMathOperator{\mim}{Im} 
\DeclareMathOperator{\leng}{length}

\newcommand{\sigmac}{\sigma_{\mathrm{c}}} 
\newcommand{\sigmaa}{\sigma_{\mathrm{a}}}

\begin{document} 
\title[Carlson's theorem and vertical limit functions]{Carlson's theorem and \\ vertical limit functions} 
\date{\today}

\author{Ole Fredrik Brevig} 
\address{Department of Mathematical Sciences, Norwegian University of Science and Technology (NTNU), 7491 Trondheim, Norway} 
\email{ole.brevig@ntnu.no}

\author{Athanasios Kouroupis} 
\address{Department of Mathematics, KU Leuven, Celestijnenlaan 200B, 3001, Leuven, Belgium} 
\email{athanasios.kouroupis@kuleuven.be}

\begin{abstract}
	We extend a classical theorem of Carlson on moments of Dirichlet series from $p=2$ to $1 \leq p < \infty$. When combined with the ergodic theorem for the Kronecker flow, a coherent approach to almost sure properties of vertical limit functions in $H^p$ spaces of Dirichlet series is obtained. This allows us to establish an almost sure analytic continuation of vertical limit functions to the right half-plane that can be used to compute the $H^p$ norm and to prove a version of Fatou's theorem.
\end{abstract}


\thanks{Brevig is supported by Grant 354537 of the Research Council of Norway. Kouroupis is supported by Grant 1203126N of the Research Foundation -- Flanders (FWO)}

\maketitle

\section{Introduction} For fixed $1 \leq p < \infty$, let $\mathscr{H}^p$ stand for the closure of the set of (ordinary) Dirichlet polynomials $f(s) = \sum_{n=1}^N a_n n^{-s}$ in the norm defined by 
\begin{equation}\label{eq:Hpnorm} 
	\|f\|_p = \sup_{\sigma>0} M_p(\sigma,f), 
\end{equation}
where 
\begin{equation}\label{eq:pmeans} 
	M_p^p(\sigma,f) = \lim_{T\to\infty} \frac{1}{2T} \int_{-T}^T |f(\sigma+it)|^p \,dt. 
\end{equation}
The theory of almost periodic functions ensures that the limit \eqref{eq:pmeans} exists and is nonzero (unless $f$ is identically zero). We also have the analogue of Hardy's convexity theorem in this context (see Theorem~\ref{thm:hardy} below), which allows us to replace the supremum over $\sigma$ in \eqref{eq:Hpnorm} by the limit $\sigma \to 0^+$. 

From this point of view, it is not entirely unreasonable to expect that the elements of $\mathscr{H}^p$ should be analytic functions in the right half-plane $\mathbb{C}_0$, where we write $\mathbb{C_\kappa} = \{s = \sigma + it \,:\, \sigma>\kappa\}$. 

However, it is well-known that elements of $\mathscr{H}^p$ may be represented as (absolutely) convergent Dirichlet series in the half-plane $\mathbb{C}_{1/2}$ and, moreover, that there are elements of $\mathscr{H}^p$ that do not admit an analytic continuation to any larger set (see e.g. \cite{QQ2020}*{Corollary~8.4.1}). This leads to the dichotomy that the two half-planes $\mathbb{C}_0$ and $\mathbb{C}_{1/2}$ both have important roles to play in the theory. The purpose of the present paper is to delineate precisely how the elements of $\mathscr{H}^p$ may be extended to analytic functions in $\mathbb{C}_0$ and explain the interaction between this extension and the norm \eqref{eq:Hpnorm}.

Our first result is a sufficient condition for membership in $\mathscr{H}^p$ for a somewhere convergent Dirichlet series $f(s) = \sum_{n\geq1} a_n n^{-s}$ that admits an analytic continuation to $\mathbb{C}_0$. 
\begin{theorem}\label{thm:carlson} 
	Fix $1 \leq p < \infty$. If $f$ is a somewhere convergent Dirichlet series that has an analytic continuation to $\mathbb{C}_0$ satisfying 
	\begin{equation}\label{eq:supsup} 
		\sup_{\sigma>0} \sup_{T\geq1} \frac{1}{2T} \int_{-T}^T |f(\sigma+it)|^p < \infty, 
	\end{equation}
	then 
	\begin{enumerate}
		\item[(i)] the limit $M_p(\sigma,f)$ exists for every $0<\sigma<\infty$, 
		\item[(ii)] the function $\sigma \mapsto \log{M_p(\sigma,f)}$ is decreasing and convex, 
		\item[(iii)] $f$ is in $\mathscr{H}^p$ and
		\[\|f\|_p = \lim_{\sigma \to 0^+} M_p(\sigma,f).\]
	\end{enumerate}
\end{theorem}

Here (ii) is the analogue of Hardy's convexity theorem in our setting. One should interpret the statements (i) and (ii) as saying that although $f$ may not be almost periodic in $\mathbb{C}_0$ (it may not even be bounded!), the condition \eqref{eq:supsup} allows us to transfer these properties from the half-plane of absolute convergence to $\mathbb{C}_0$.

Theorem~\ref{thm:carlson} is an extension of an old result due to Carlson~\cite{Carlson1922} (see also Titchmarsh~\cite{Titchmarsh1958}*{\S 9.51}), which corresponds to the case $p=2$. Our proof of (i) is similar to Carlson's proof in that it relies on approximating the Dirichlet series $f(s) = \sum_{n\geq1} a_n n^{-s}$ by its Riesz means. This approximation also provides the proof of (ii). The case $p=2$ is particularly favorable in regards to the assertion (iii), since straight-forward computations reveal that
\[\|f\|_2^2 = \sum_{n=1}^\infty |a_n|^2 \qquad \text{and} \qquad M(\sigma,f) = \sum_{n=1}^\infty |a_n|^2 n^{-2\sigma}.\]
In the general case such a formula is not available to us. We are therefore forced to rely on the connection between function theory of Dirichlet series and Fourier analysis on the infinite-dimensional torus $\mathbb{T}^\infty = \mathbb{T} \times \mathbb{T} \times \mathbb{T} \times \cdots$, where $\mathbb{T}$ denotes the unit circle in the complex plane.

The key idea in this context (which goes essentially back to Bohr~\cite{Bohr1913}) is that the \emph{Kronecker flow}
\[\mathfrak{p}^{-i\tau} = (2^{-i\tau},3^{-i\tau},5^{-i\tau},\ldots),\]
for $\tau$ in $\mathbb{R}$, is ergodic on $\mathbb{T}^\infty$ with respect to its Haar measure $m_\infty$. The dual group of $\mathbb{T}^\infty$ can be identified with $\mathbb{Q}_+$. Due to the multiplicative structure, this means that any character $\chi$ on $\mathbb{T}^\infty$ is uniquely determined by its value at the prime numbers $\chi(p_j) = \chi_j$ for $j=1,2,3,\ldots$. If $f(s) = \sum_{n=1}^N a_n n^{-s}$ and $f^\ast(\chi) = \sum_{n=1}^N a_n \chi(n)$, then plainly $f(i\tau) = f^\ast(\mathfrak{p}^{-i\tau})$. Appealing to the ergodic theorem (see e.g.~\cite{QQ2020}*{Chapter 2}), we may infer that 
\begin{equation}\label{eq:ergodicpoly} 
	\|f\|_p^p = \lim_{T\to\infty} \frac{1}{2T} \int_{-T}^T |f(i\tau)|^p\,d\tau = \int_{\mathbb{T}^\infty} |f^\ast(\chi)|^p \,dm_\infty(\chi) = \|f^\ast\|_{L^p(\mathbb{T}^\infty)}^p. 
\end{equation}
This demonstrates that $\mathscr{H}^p$---which we abstractly defined as the closure of the set of Dirichlet polynomials in the norm \eqref{eq:Hpnorm} above---is isometrically isomorphic to the Banach space
\[H^p(\mathbb{T}^\infty) = \left\{f^\ast \in L^p(\mathbb{T}^\infty)\,:\, \widehat{f^\ast}(q) = 0 \,\text{ if }\, q \in \mathbb{Q}_+ \setminus \mathbb{N}\right\}.\]
If $f^\ast$ is a function in $L^p(\mathbb{T}^\infty)$ that is not necessarily continuous, then the ergodic theorem asserts that 
\begin{equation}\label{eq:ergodicgen} 
	\|f^\ast\|_{L^p(\mathbb{T}^\infty)}^p = \lim_{T\to\infty} \int_{-T}^T |f^\ast(\chi \mathfrak{p}^{-i\tau})|^p\,d\tau 
\end{equation}
holds for almost every $\chi$ on $\mathbb{T}^\infty$. (Note that in \eqref{eq:ergodicpoly}, the function $f^\ast$ is continuous so we may pick \emph{any} $\chi$: in particular $\chi \equiv 1$.) The set of $\chi$ such that \eqref{eq:ergodicgen} holds will play an important role in this paper.
\begin{definition}
	Fix $1 \leq p < \infty$ and suppose that $f$ is an element of $\mathscr{H}^p$. We let $E_f$ denote the subset of $\mathbb{T}^\infty$ of full measure enjoying the property that if $\chi$ is in $E_f$, then
	\[\|f\|_p^p = \lim_{T\to\infty} \frac{1}{2T} \int_{-T}^T |f^\ast(\chi \mathfrak{p}^{-i\tau})|^p \,d\tau.\]
	We also let $C = C_f(\chi)$ stand for the smallest positive number such that
	\[\int_{-T}^0 |f^\ast(\chi \mathfrak{p}^{-i\tau})|^p \, d\tau \leq C(1+T) \quad\text{and}\quad \int_0^T |f^\ast(\chi \mathfrak{p}^{-i\tau})|^p \, d\tau \leq C(1+T)\]
	holds for all $T>0$.
\end{definition}

Let us return now to the relationship between $\mathscr{H}^p$ and $H^p(\mathbb{T}^\infty)$. If $f^\ast$ is a function in $L^1(\mathbb{T}^\infty)$, then the work of Cole and Gamelin~\cite{CG1986} asserts that the Poisson integral 
\begin{equation}\label{eq:poissonTinfty} 
	\mathfrak{P}f^\ast(z) = \int_{\mathbb{T}^\infty} f^\ast(\chi) \prod_{j=1}^\infty \frac{1-|z_j|^2}{|1-\overline{\chi_j} z_j|^2}\,dm_\infty(\chi) 
\end{equation}
is in general well-defined if and only if $z$ is a point in $\mathbb{D}^\infty \cap \ell^2$, where $\mathbb{D}$ denotes the unit disc in the complex plane. If $f^\ast$ is in $H^1(\mathbb{T}^\infty)$, then Helson's inequality \cite{Helson2006} implies that $\mathfrak{P}f^\ast$ can be expressed as an absolutely convergent monomial series for $z$ in $\mathbb{D}^\infty \cap \ell^2$. In particular, we may identify each $f$ in $\mathscr{H}^p$ with the absolutely convergent Dirichlet series $f(s) = \mathfrak{P}f^\ast(\mathfrak{p}^{-s})$ if $\mathfrak{p}^{-s}$ is in $\mathbb{D}^\infty \cap \ell^2$ or, equivalently, if $s$ is in $\mathbb{C}_{1/2}$.

What we have just described realizes $\mathscr{H}^p$ as a Banach space of Dirichlet series that converge absolutely in $\mathbb{C}_{1/2}$. Absolutely convergent Dirichlet series are almost periodic, so it follows that every sequence of vertical translations $V_\tau f(s) = f(s+i\tau)$ will have a subsequence that converges uniformly any strictly smaller half-plane. In our setting, these vertical limit functions take the form
\[f_\chi(s) = \sum_{n=1}^\infty a_n \chi(n) n^{-s}\]
for $\chi$ on $\mathbb{T}^\infty$ due to Kronecker's theorem (see \cite{HLS1997}*{Lemma~2.4}). As mentioned above, there are elements of $\mathscr{H}^p$ that cannot be analytically continued to any set strictly containing $\mathbb{C}_{1/2}$. To circumvent this obstacle we consider not the function $f$ but its collection of vertical limit functions $f_\chi$.
\begin{theorem}\label{thm:ergodicpoisson} 
	Fix $1 \leq p < \infty$. If $f$ is an element of $\mathscr{H}^p$ and if $\chi$ in $E_f$, then the Poisson integral 
	\begin{equation}\label{eq:poissonhalfplane} 
		s \mapsto \int_{-\infty}^\infty f^\ast(\chi \mathfrak{p}^{-i\tau}) \, \frac{\sigma}{\sigma^2 + (t-\tau)^2}\,\frac{d\tau}{\pi} 
	\end{equation}
	defines an analytic function in $\mathbb{C}_0$ coinciding with $f_\chi$ in $\mathbb{C}_{1/2}$. Moreover, if $f_\chi$ denotes this analytic continuation, then 
	\begin{equation}\label{eq:supsupchi} 
		\sup_{\sigma>0} \sup_{T\geq1} \frac{1}{2T} \int_{-T}^T |f_\chi(\sigma+it)|^p \leq 6 C_f(\chi). 
	\end{equation}
\end{theorem}

It is well-known that $f_\chi$ has an analytic continuation to $\mathbb{C}_0$ for almost every $\chi$ on $\mathbb{T}^\infty$, and this can be established by several distinct methods (see e.g.~\cites{Bayart2002,HLS1997,Helson1969}). The novelty of our approach lies in the comparison of Poisson integral \eqref{eq:poissonTinfty} for $z = \mathfrak{p}^{-s}$ and $s$ in $\mathbb{C}_{1/2}$ and the Poisson integral \eqref{eq:poissonhalfplane} for $\chi$ in $E_f$. This allows us to extend the ergodic statement on the imaginary axis \eqref{eq:ergodicgen} into $\mathbb{C}_0$ and obtain \eqref{eq:supsupchi}.

It follows from \eqref{eq:supsupchi} that if $f$ is in $\mathscr{H}^p$ and $\chi$ is in $E_f$, then the function $f_\chi(s)/(s+1/2)^{2/p}$ is in the Hardy space $H^p(\mathbb{C}_0)$. This means that Fatou's theorem (see e.g.~\cite{Garnett2007}*{Chapter~II.3}) holds for $f_\chi$, so that the limit 
\begin{equation}\label{eq:fatouclassical} 
	f^\ast(\chi \mathfrak{p}^{-i\tau}) = \lim_{\sigma \to 0^+} f_\chi(\sigma+it) 
\end{equation}
exists for almost every $\tau$ in $\mathbb{R}$. We let $f_\chi(i\tau)$ denote this limit (when it exists). Combining \eqref{eq:fatouclassical} with Theorem~\ref{thm:carlson} and Theorem~\ref{thm:ergodicpoisson}, we obtain the following result.
\begin{corollary}\label{cor:norms} 
	Fix $1 \leq p < \infty$. If $f$ is in $\mathscr{H}^p$ and if $\chi$ in $E_f$, then
	\[\|f\|_p^p = \lim_{\sigma \to 0^+} \lim_{T\to\infty} \frac{1}{2T} \int_{-T}^T |f_\chi(\sigma+it)|^p dt = \lim_{T\to\infty} \frac{1}{2T}\int_{-T}^T |f_\chi(i \tau)|^p \,d\tau.\]
\end{corollary}

The key point of this result is that the same set $E_f$ guarantees the existence of $M_p^p(\sigma,f_\chi)$ for all $0 \leq \sigma < \infty$. It is interesting to compare Corollary~\ref{cor:norms} with the work of Saksman and Seip~\cite{SS2009}, which concerns the space $\mathscr{H}^\infty$ of Dirichlet series with bounded analytic continuations to $\mathbb{C}_0$. In this case, the first equality in Corollary~\ref{cor:norms} holds for every $\chi$, since the elements of $\mathscr{H}^\infty$ are almost periodic in $\mathbb{C}_\kappa$ for every $\kappa>0$. However, Saksman and Seip constructed a Dirichlet series in $\mathscr{H}^\infty$ exemplifying that the second equality in Corollary~\ref{cor:norms} may fail in a rather spectacular manner for $\chi \equiv 1$.

Fatou's theorem in the context of $\mathscr{H}^p$ theory now follows from another application of the ergodic theorem in combination with the first assertion of Theorem~\ref{thm:ergodicpoisson} and \eqref{eq:fatouclassical}. 
\begin{theorem}\label{thm:fatou} 
	Fix $1 \leq p < \infty$. If $f$ is an element of $\mathscr{H}^p$, then there is a subset $E$ of $\mathbb{T}^\infty$ of full measure such that if $\chi$ is in $E$, then
	\[f^\ast(\chi) = \lim_{\sigma \to 0^+} f_\chi(\sigma).\]
\end{theorem}

Theorem~\ref{thm:fatou} was established by Saksman and Seip \cite{SS2009} for $\mathscr{H}^\infty$. Our argument is similar to theirs, but since the elements of $\mathscr{H}^\infty$ are bounded analytic functions in $\mathbb{C}_0$, the machinery developed in Theorem~\ref{thm:ergodicpoisson} is not required in this case. 

It is possible (see e.g.~\cite{Bayart2002}*{Theorem 6}) to prove that $f_\chi$ has an analytic continuation to $\mathbb{C}_0$ for almost every $\chi$ by demonstrating that the Dirichlet series $f_\chi$ in fact converges in $\mathbb{C}_0$. The next point we wish to make is that the analytic continuations obtained from the Poisson integral \eqref{eq:poissonhalfplane} also enjoy this property. We will rely on the following result.
\begin{theorem}[Titchmarsh~\cite{Titchmarsh1958}*{\S9.55}] \label{thm:titchmarsh} 
	If $f$ is a somewhere convergent Dirichlet series that has an analytic continuation to $\mathbb{C}_0$ satisfying
	\[\sup_{\sigma>0} \sup_{T\geq1} \frac{1}{2T} \int_{-T}^T |f(\sigma+it)|^2 < \infty,\]
	then the Dirichlet series $f$ converges in $\mathbb{C}_0$. 
\end{theorem}

The combination of Theorem~\ref{thm:ergodicpoisson} and Theorem~\ref{thm:titchmarsh} yields at once that if $f$ is in $\mathscr{H}^p$ for some $p\geq2$ and $\chi$ is in $E_f$, then $f_\chi$ converges in $\mathbb{C}_0$. This can be extended to the case $p\geq1$ using Helson's inequality \cite{Helson2006}, which asserts that if $f(s) = \sum_{n\geq1} a_n n^{-s}$ is in $\mathscr{H}^1$, then 
\begin{equation}\label{eq:helsonineq} 
	\left(\sum_{n=1}^\infty \frac{|a_n|^2}{d(n)}\right)^{\frac{1}{2}} \leq \|f\|_1, 
\end{equation}
where $d(n)$ denotes the number of divisors of the integer $n$. Since $d(n) \leq C_\epsilon n^{\varepsilon}$ for every $\varepsilon>0$, it follows from Helson's inequality that if $f$ is in $\mathscr{H}^1$, then the horizontal translation $H_\kappa f(s) = f(s+\kappa)$ belongs to $\mathscr{H}^2$ for every $\kappa>0$.
\begin{corollary}
	Fix $1 \leq p < \infty$. If $f$ is in $\mathscr{H}^p$ and if $\chi$ in $E_f$, then $f_\chi$ converges in $\mathbb{C}_0$. 
\end{corollary}

\subsection*{Organization} The present paper is comprised of two additional sections. Section~\ref{sec:carlson} contains some preliminary material on Riesz means and culminates with the proof of Theorem~\ref{thm:carlson}. The proofs of Theorem~\ref{thm:ergodicpoisson} and Theorem~\ref{thm:fatou} can be found in Section~\ref{sec:almostsure}.

\section{An extension of Carlson's theorem} \label{sec:carlson} 
We begin by demonstrating that the condition \eqref{eq:supsup} provides a pointwise bound on $f$. The proof is standard (see e.g. Titchmarsh~\cite{Titchmarsh1958}*{\S 9.55}), but we include the details since some formulations of Carlson's theorem in the recent literature include both the condition \eqref{eq:supsup} and an assumption that $f$ has finite order in $\mathbb{C}_0$.
\begin{lemma}\label{lem:pest} 
	If $f$ is an analytic function in $\mathbb{C}_0$ enjoying the property that
	\[C_f = \sup_{\sigma>0} \sup_{T\geq1} \frac{1}{2T}\int_{-T}^T |f(\sigma+it)|^p \,dt < \infty,\]
	for some $1 \leq p < \infty$, then
	\[|f(s)|^p \leq 2 C_f \frac{1+|s|}{\sigma}\]
	for every $s=\sigma+it$ in $\mathbb{C}_0$. 
\end{lemma}
\begin{proof}
	Fix a point $s=\sigma+it$ in $\mathbb{C}_0$ and assume without loss of generality that $t \geq 0$. By the sub-mean value property, we have that 
	\begin{multline*}
		|f(s)|^p \leq \frac{1}{\pi\sigma^2}\int\limits_{|z-s| \leq \sigma} |f(x+iy)|^p \,dxdy \\
		\leq \frac{1}{\pi\sigma^2}\int_0^{2\sigma} \int_{-(t+\sigma+1)}^{t+\sigma+1} |f(x+iy)|^p \,dydx \leq \frac{4}{\pi \sigma} C_f (t+\sigma+1). 
	\end{multline*}
	The proof is completed by using that $t+\sigma \leq \sqrt{2}|s|$ and $4\sqrt{2}/\pi \leq 2$. 
\end{proof}

The following technical estimate is crucial for the proof of Theorem~\ref{thm:carlson}. 
\begin{lemma}\label{lem:avgest} 
	Let $f$ be an analytic function in $\mathbb{C}_0$ enjoying the property that
	\[C_f = \sup_{\sigma>0} \sup_{T\geq1} \frac{1}{2T}\int_{-T}^T |f(\sigma+it)|^p \,dt < \infty,\]
	for some $1 \leq p < \infty$. If $T\geq1$ and if $z=x+iy$ is a point in $\mathbb{C}_0$, then
	\[\left(\frac{1}{2T}\int_{-T}^T |f(\sigma+it+z)-f(\sigma+it)|^p \,dt\right)^{1/p} \leq 3 (C_f)^{1/p} \frac{|z|}{\sigma^2}\big(1+\sigma+|z|\big)^{1/p+1}.\]
\end{lemma}
\begin{proof}
	Let $\Gamma$ denote the rectangle with vertices at
	\[z + \frac{\sigma}{2}(1+i), \qquad iy+\frac{\sigma}{2}(-1+i), \qquad -\frac{\sigma}{2}(1+i), \qquad x - \frac{\sigma}{2}(-1+i),\]
	oriented counterclockwise. The points $0$ and $z$ lie in the interior of $\Gamma$. Hence
	\[f(\sigma+it+z)-f(\sigma+it) = \oint_\Gamma f(\sigma+it+\xi) \left(\frac{1}{\xi-z}-\frac{1}{\xi}\right)\,\frac{d\xi}{2\pi i},\]
	where $\sigma+it+\xi$ remains within $\mathbb{C}_0$ since $\mre{\xi} \geq \sigma/2$. Crashing through with absolute values and Minkowski's inequality, we get 
	\begin{multline*}
		\left(\frac{1}{2T}\int_{-T}^T |f(\sigma+it+z)-f(\sigma+it)|^p \,dt\right)^{1/p} \\
		\leq \oint_\Gamma \left|\frac{1}{\xi-z}-\frac{1}{\xi}\right| \left(\frac{1}{2T} \int_{-T}^T |f(\sigma+it+\xi)|^p \,dt\right)^{1/p} \, \frac{d|\xi|}{2\pi}. 
	\end{multline*}
	Arguing as in the proof of Lemma~\ref{lem:pest}, we find that if $T\geq1$ and if $\xi$ is on $\Gamma$, then 
	\begin{equation}\label{eq:est1} 
		\frac{1}{2T} \int_{-T}^T |f(\sigma+it+\xi)|^p \,dt \leq C_f\frac{T+|\mim{\xi}|}{T} \leq C_f(1+\sigma+|z|). 
	\end{equation}
	Using the trivial estimate $|\xi-z|,\,|\xi|\geq \sigma/2$ for $\xi$ on $\Gamma$, we infer that 
	\begin{equation}\label{eq:est2} 
		\oint_\Gamma \left|\frac{1}{\xi-z}-\frac{1}{\xi}\right|\,\frac{d|\xi|}{2\pi} = |z| \oint_\Gamma \left|\frac{1}{(\xi-z)\xi}\right|\,\frac{d|\xi|}{2\pi} \leq \frac{2|z|}{\pi \sigma^2} \leng(\Gamma). 
	\end{equation}
	All that remains to complete the proof is to combine \eqref{eq:est1} and \eqref{eq:est2}, before noting that $\leng{\Gamma} = 2x+2y+4\sigma \leq 4(1+\sigma+|z|)$ and that $8/\pi \leq 3$. 
\end{proof}

It is well-known (see e.g.~\cite{QQ2020}*{Chapter~4.2}) that if $f(s) = \sum_{n\geq1} a_n n^{-s}$ is a somewhere convergent Dirichlet series, then there is a number $\sigmac(f)<\infty$ (that may equal $-\infty$) called the \emph{abscissa of convergence} with the property $f$ converges if and only if $\mre{s}>\sigmac(f)$. There is a similarly defined \emph{abscissa of absolute convergence} denoted $\sigmaa(f)$. From our point of view the classical estimate $\sigmaa(f)-\sigmac(f) \leq 1$ will be important, since it demonstrates that every somewhere convergent Dirichlet series is somewhere absolutely convergent. 

We will use the \emph{Riesz means} 
\begin{equation}\label{eq:rieszmeans} 
	R_N^k f(s) = \sum_{n=1}^N a_n \left(1-\frac{\log{n}}{\log{N}}\right)^k n^{-s} 
\end{equation}
in the proof of Theorem~\ref{thm:carlson}. Note that if $k>0$ is fixed, then $R_N^k f$ plainly converges pointwise to $f$ in $\mathbb{C}_{\sigmaa(f)}$ as $N\to\infty$.

The next result is essentially due to Riesz (see Hardy and Riesz~\cite{HR1964}*{\S VII.4}), but the Riesz means \eqref{eq:rieszmeans} are of \emph{first kind} as opposed to Riesz means of \emph{second kind} used in \cites{Carlson1922,HR1964}. The details of the proof are therefore slightly different and we include the full account for the benefit of the reader.
\begin{theorem}\label{thm:rieszac} 
	If $f$ is a somewhere convergent Dirichlet series that has an analytic continuation to $\mathbb{C}_0$ enjoying the pointwise estimate 
	\begin{equation}\label{eq:simplepest} 
		|f(s)| \leq C \frac{1+|s|}{\sigma}, 
	\end{equation}
	and if $N\geq2$, $k>1$, $\sigma>0$, and $z=x+iy$ for $x>0$, then 
	\begin{equation}\label{eq:rieszac} 
		R_N^k f(s) = \Gamma(k+1)\int_{-\infty}^\infty f\left(s+\frac{z}{\log{N}}\right) \frac{e^z}{z^{k+1}}\,\frac{dy}{2\pi}. 
	\end{equation}
	Moreover, $R_N^k f$ converges pointwise to $f$ in $\mathbb{C}_0$. 
\end{theorem}
\begin{proof}
	The starting point is the formula 
	\begin{equation}\label{eq:hankel} 
		\Gamma(k+1)\int_{x-i\infty}^{x+\infty} \frac{e^{u\xi}}{\xi^{k+1}}\,\frac{d\xi}{2\pi i} = 
		\begin{cases}
			u^k & \text{if } u \geq 0, \\
			0 & \text{if } u < 0, 
		\end{cases}
	\end{equation}
	which is Hankel's formula for $1/\Gamma(k+1)$, valid for $k>0$ and $x>0$. If $f$ is a somewhere convergent Dirichlet series, then we may apply \eqref{eq:hankel} with $u=\log(N/n)$ for $n=1,2,3,\ldots$ to obtain the smoothed Perron--type formula
	\[\sum_{n=1}^N a_n (\log{N}-\log{n})^k n^{-s} = \Gamma(k+1)\int_{x-i\infty}^{x+i\infty} f(s+\xi) \frac{N^\xi}{\xi^{k+1}}\,\frac{d\xi}{2\pi i},\]
	so long as $x > \max(0,\sigmaa(f))$. If $k>0$ and $\sigma>0$, then the validity of this identity may be extended to $x>0$ by using Cauchy's theorem on a rectangular contour and invoking the estimate \eqref{eq:simplepest} on the horizontal segments. We divide both sides by $(\log{N})^{-k}$ to obtain
	\[R_N^kf(s) = \frac{\Gamma(k+1)}{(\log{N})^k} \int_{x-i\infty}^{x+i\infty} f(s+\xi) \frac{N^\xi}{\xi^{k+1}}\,\frac{d\xi}{2\pi i}.\]
	Substituting $x \mapsto x/\log{N}$ (which is permissible since we can choose $x>0$ freely) and choosing the parametrization $\xi=(x+iy)/\log{N}$ for $y$ in $\mathbb{R}$, we obtain the formula \eqref{eq:rieszac}. Let us now handle the final assertion. Using \eqref{eq:rieszac} and \eqref{eq:hankel} with $u=1$, we obtain 
	\begin{equation}\label{eq:RNkff} 
		|R_N^k f(s)-f(s)| \leq \Gamma(k+1)\int_{-\infty}^\infty \left|f\left(s+\frac{z}{\log{N}}\right)-f(s)\right| \frac{e^x}{|z|^{k+1}}\,\frac{dy}{2\pi}. 
	\end{equation}
	Fixing some $x>0$ and $s$ in $\mathbb{C}_0$, we obtain the stated result from the dominated convergence theorem due to the pointwise estimate \eqref{eq:simplepest} and the assumption that $k>1$. 
\end{proof}
\begin{remark}
	The argument used in proof of Theorem~\ref{thm:rieszac} shows that if $f$ is a somewhere convergent Dirichlet series with an analytic continuation to $\mathbb{C}_0$ of zero order, then $\sigmac(f) \leq 0$. See Titchmarsh~\cite{Titchmarsh1958}*{\S9.44}. 
\end{remark}

In preparation for the proof of Theorem~\ref{thm:carlson}, let us collect two additional results. The first is the following special case of Hardy's convexity theorem for almost periodic functions (see \cite{HW1941} or \cite{BP2021}*{Section~3}).
\begin{theorem}\label{thm:hardy} 
	Fix $1 \leq p < \infty$. If $f(s) = \sum_{n\geq1} a_n n^{-s}$ converges uniformly in $\mathbb{C}_\kappa$, then the function
	\[\sigma \mapsto \log{M_p(\sigma,f)}\]
	is decreasing and convex for $\kappa<\sigma<\infty$. 
\end{theorem}

Note in particular that Theorem~\ref{thm:hardy} applies to Dirichlet polynomials (with any $\kappa$) and to somewhere convergent Dirichlet series $f$ with $\kappa = \sigmaa(f)$. 

We will also have use of the following result (see \cite{BP2021}*{Lemma~3.1}). Let us stress that we do not know a proof that does not appeal to the connection between $\mathscr{H}^p$ and $H^p(\mathbb{T}^\infty)$. Before stating the result, let us recall that $H_\kappa$ stands for the horizontal translations $H_\kappa f(s) = f(s+\kappa)$ for $\kappa>0$.
\begin{lemma}\label{lem:Hsigma} 
	Fix $1 \leq p < \infty$ and suppose that $f$ is a somewhere convergent Dirichlet series. If $H_\sigma f$ is in $\mathscr{H}^p$ for every $\sigma>0$ and
	\[\sup_{\sigma>0} \|H_\sigma f\|_p < \infty,\]
	then $f$ is $\mathscr{H}^p$ and $\|f-H_\sigma f\|_p \to 0$ as $\sigma \to 0^+$. 
\end{lemma}

We are now ready to proceed with the proof of Theorem~\ref{thm:carlson}. We will argue similarly to the final part of the proof of Theorem~\ref{thm:rieszac} to establish (i), but now Lemma~\ref{lem:avgest} will enter the picture. Theorem~\ref{thm:hardy} and Lemma~\ref{lem:Hsigma} are, respectively, required for assertions (ii) and (iii). 
\begin{proof}
	[Proof of Theorem~\ref{thm:carlson}] Fix $\sigma>0$, $N\geq2$, and $k>1$. The assumption \eqref{eq:supsup} and Lemma~\ref{lem:pest} implies that \eqref{eq:RNkff} from the proof of Theorem~\ref{thm:rieszac} holds in our setting, which when used in combination with Minkowski's inequality yields that 
	\begin{multline*}
		\left(\frac{1}{2T}\int_{-T}^T |R_N^kf(s)-f(s)|^p \,dt \right)^{1/p} \\
		\leq \Gamma(k+1) \int_{-\infty}^\infty \left(\frac{1}{2T}\int_{-T}^T \left|f\left(s+\frac{z}{\log{N}}\right)-f(s)\right|^p\,dt\right)^{1/p} \frac{e^x}{|z|^{k+1}}\,\frac{dy}{2\pi}. 
	\end{multline*}
	The inner integral can be estimated using Lemma~\ref{lem:avgest} (where $z$ is $z/\log{N}$). If $T\geq1$ and $N\geq3$ (so that $\log{N}\geq1$), we get 
	\begin{multline*}
		\left(\frac{1}{2T}\int_{-T}^T |R_N^kf(s)-f(s)|^p \,dt \right)^{1/p} \\
		\leq \frac{3 \Gamma(k+1) (C_f)^{1/p}}{ \sigma^2 \log N} \int_{-\infty}^\infty (1+\sigma+|z|)^{1/p+1} \frac{ e^x}{|z|^k}\,\frac{dy}{2\pi}, 
	\end{multline*}
	where as usual
	\[C_f = \sup_{\sigma>0} \sup_{T\geq1} \frac{1}{2T}\int_{-T}^T |f(\sigma+it)|^p \,dt.\]
	For fixed $x>0$ (we warmly recommend $x=k$), the integral on the right-hand side is finite whenever $k-1/p-1>1$, so let us also fix any $k>3$. This demonstrates that there for every $\varepsilon>0$ is a positive integer $N_{\varepsilon}$ such that if $N \geq N_\varepsilon$, then 
	\begin{equation}\label{eq:Tunieps} 
		\sup_{T\geq1}\left(\frac{1}{2T}\int_{-T}^T |R_N^kf(\sigma+it)-f(\sigma+it)|^p \,dt \right)^{1/p} \leq \varepsilon 
	\end{equation}
	for fixed $0<\sigma<\infty$. From this, it is difficult not to see that 
	\begin{equation}\label{eq:ffromRnkf} 
		\lim_{T\to\infty} \frac{1}{2T} \int_{-T}^T |f(\sigma+it)|^p \,dt = \lim_{N\to\infty} \lim_{T\to\infty} \frac{1}{2T} \int_{-T}^T |R_N^kf(\sigma+it)|^p \,dt 
	\end{equation}
	in the sense that the $T$-limit on the left-hand side and the $N$-limit on the right-hand side both exist, are finite, and coincide. This completes the proof of (i). The proof of (ii) follows from \eqref{eq:ffromRnkf} and Theorem~\ref{thm:hardy}, since pointwise limits of decreasing and convex functions are decreasing and convex.
	
	We turn next to (iii) and suppose that $f(s) = \sum_{n\geq1} a_n n^{-s}$. It follows from the estimate \eqref{eq:Tunieps} and Theorem~\ref{thm:hardy} (applied to the Dirichlet polynomial $H_\sigma R_{N_1}^kf-H_\sigma R_{N_2}^k f$) that $(H_\sigma R_N^kf)_{N\geq3}$ is a Cauchy sequence in $\mathscr{H}^p$. Since a function in $H^p(\mathbb{T}^\infty)$ is uniquely determined by its Fourier coefficients, it follows that
	\[(H_\sigma f)^\ast(\chi) = \sum_{n=1}^\infty a_n n^{-\sigma} \chi(n)\]
	is in $H^p(\mathbb{T}^\infty)$ and that $\|(H_\sigma f)^\ast\|_{L^p(\mathbb{T}^\infty)} = M_p(\sigma,f)$ due to (ii). Using the Poisson integral \eqref{eq:poissonTinfty}, we find that $H_\sigma f$ is in $\mathscr{H}^p$ for every $\sigma>0$ and that
	\[\|H_\sigma f\|_p = M_p(\sigma,f).\]
	Using that $C_f<\infty$, it now follows from Lemma~\ref{lem:Hsigma} that $f$ is in $\mathscr{H}^p$ and that
	\[\|f\|_p = \lim_{\sigma \to 0^+} \|H_\sigma f\|_p = \lim_{\sigma \to 0^+} M_p(\sigma,f). \qedhere\]
\end{proof}
\begin{remark}
	It is possible to replace the assumption \eqref{eq:supsup} in Theorem~\ref{thm:carlson} with the weaker assumption 
	\begin{equation}\label{eq:limsup} 
		\sup_{\sigma>0} \limsup_{T\rightarrow\infty} \frac{1}{2T} \int_{-T}^T |f(\sigma+it)|^p < \infty, 
	\end{equation}
	provided $f$ is also assumed to have finite order in $\mathbb{C}_\kappa$ for every $\kappa>0$. The basic idea is that the function
	\[F_T(z) = \frac{1}{2T}\int_{-T}^T \left|\frac{f(z+it)}{z+1}\right|^p\,dt\]
	is subharmonic and of finite order in the strip $S_\kappa =\{z\,:\, \kappa < x < \sigmaa(f)+1\}$ for every fixed $\kappa>0$ and bounded on the boundary lines of the strip. By the Phragm\'{e}n--Lindel\"of principle, there is a constant $C_\kappa>0$ independent of $T \geq1$ such that $|F_T(z)| \leq C_\kappa$. Using this estimate with $y=0$, it follows that $H_\kappa f$ satisfies the assumptions of Theorem~\ref{thm:carlson} for every $\kappa>0$. Since $\|H_\kappa f\|_p$ is bounded by the quantity in \eqref{eq:limsup}, we obtain the conclusion of Theorem~\ref{thm:carlson}. 
\end{remark}

\section{Almost sure properties of vertical limit functions} \label{sec:almostsure} 
A technical result is needed for the proof of Theorem~\ref{thm:ergodicpoisson}. It can be established in various ways, but the cleanest is via Helson's inequality \eqref{eq:helsonineq}, which when formulated on $\mathbb{T}^\infty$ asserts that 
\begin{equation}\label{eq:helsonreform} 
	\left(\sum_{n=1}^\infty \frac{|\widehat{f}(n)|^2}{d(n)}\right)^{1/2} \leq \|f\|_{H^1(\mathbb{T}^\infty)} 
\end{equation}
for every $f$ in $H^1(\mathbb{T}^\infty)$.
\begin{lemma}\label{lem:helsonabs} 
	If $f$ is in $H^1(\mathbb{T}^\infty)$ and if $g$ is defined by $\widehat{g}(n) = \widehat{f}(n) z(n)$ for $z$ in $\mathbb{D}^\infty \cap \ell^2$, then the Fourier series
	\[g(\chi) = \sum_{n=1}^\infty \widehat{g}(n) \chi(n)\]
	converges absolutely on $\mathbb{T}^\infty$. 
\end{lemma}
\begin{proof}
	This follows at once from the Cauchy--Schwarz inequality, \eqref{eq:helsonreform}, and the identity
	\[\sum_{n=1}^\infty d(n) |z(n)|^2 = \prod_{j=1}^\infty \left(\frac{1}{1-|z_j|^2}\right)^2,\]
	since $|\chi(n)|=1$. 
\end{proof}

We will apply Lemma~\ref{lem:helsonabs} with $z = \mathfrak{p}^{-s}$ for $s$ in $\mathbb{C}_{1/2}$ in order to compare the two Poisson integrals \eqref{eq:poissonTinfty} and \eqref{eq:poissonhalfplane}. 
\begin{proof}
	[Proof of Theorem~\ref{thm:ergodicpoisson}] Our first task is to show that function $F$ defined by the Poisson integral
	\[F(\chi,s) = \int_{-\infty}^\infty f^\ast(\chi \mathfrak{p}^{-i\tau}) \, \frac{\sigma}{\sigma^2 + (t-\tau)^2}\,\frac{d\tau}{\pi}\]
	is well-defined for $\chi$ in $E_f$ and $s$ in $\mathbb{C}_0$. We begin by using H\"older's inequality to the effect that
	\[\int_{-\infty}^\infty |f^\ast(\chi \mathfrak{p}^{-i\tau})| \, \frac{\sigma}{\sigma^2 + (t-\tau)^2}\,\frac{d\tau}{\pi} \leq \left(\int_{-\infty}^\infty |f^\ast(\chi \mathfrak{p}^{-i\tau})|^p \, \frac{\sigma}{\sigma^2 + (t-\tau)^2}\,\frac{d\tau}{\pi}\right)^{\frac{1}{p}}.\]
	Suppose that $\chi$ is in $E_f$ and let $C=C_f(\chi)$ denote the constant from the definition of $E_f$. Using integration by parts, we find that 
	\begin{equation}\label{eq:pest1} 
		\int_{-\infty}^\infty |f^\ast(\chi \mathfrak{p}^{-i\tau})|^p \, \frac{\sigma}{\sigma^2 + (t-\tau)^2}\,\frac{d\tau}{\pi} \leq 2C \int_{-\infty}^\infty \frac{(1+|\tau|) \sigma |t-\tau|}{\left(\sigma^2 + (t-\tau)^2\right)^2}\,\frac{d\tau}{\pi}. 
	\end{equation}
	Using that $|\tau| \leq |t|+|t-\tau|$ and computing the resulting integrals, we obtain the estimate 
	\begin{equation}\label{eq:pest2} 
		\int_{-\infty}^\infty \frac{(1+|\tau|) \sigma |t-\tau|}{\left(\sigma^2 + (t-\tau)^2\right)^2}\,\frac{d\tau}{\pi} \leq \frac{1+|t|}{\pi \sigma} + \frac{1}{2}. 
	\end{equation}
	We infer from \eqref{eq:pest1} and \eqref{eq:pest2} that $F$ is well-defined for $\chi$ in $E_f$ and $s$ in $\mathbb{C}_0$. 
	
	Let us now fix $s$ in $\mathbb{C}_0$. Since $E_f$ has full measure in $\mathbb{T}^\infty$, we may think of $F(\cdot,s)$ as a function defined almost everywhere on $\mathbb{T}^\infty$. Our next task is to show that $F(\cdot,s)$ is in $L^1(\mathbb{T}^\infty)$. We use H\"older's inequality twice and Tonelli's theorem with the rotational invariance of $m_\infty$ to infer that 
	\begin{multline*}
		\int_{\mathbb{T}^\infty} |F(\chi,s)| \,dm_\infty(\chi) \\
		\leq \int_{-\infty}^\infty \left(\int_{\mathbb{T}^\infty} |f^\ast(\chi \mathfrak{p}^{-i\tau})|^p \,dm_\infty(\chi)\right)^{\frac{1}{p}} \frac{\sigma}{\sigma^2 + (\tau-t)^2}\,\frac{d\tau}{\pi} = \|f\|_p. 
	\end{multline*}
	Since $q^{-s}$ for $q \geq 1$ and $q^{s}$ for $0<q<1$ are bounded analytic functions in $\mathbb{C}_0$, we get that
	\[\int_{-\infty}^\infty \frac{q^{-i\tau}}{\sigma^2+(t-\tau)^2}\,\frac{d\tau}{\pi} = \frac{q^{-it}}{\left(\max(q,1/q)\right)^\sigma}\]
	for $q>0$. When used in conjunction with Fubini's theorem, this allows us to compute the Fourier coefficients
	\[\widehat{F}(q,s) = \int_{\mathbb{T}^\infty} F(\chi,s) \,\overline{q(\chi)}\,dm_\infty(\chi) = \widehat{f^\ast}(q) \frac{q^{-it}}{\left(\max(q,1/q)\right)^\sigma}\]
	for $q$ in $\mathbb{Q}_+$. Since $f$ is in $\mathscr{H}^p$ by assumption, we know that $\widehat{f^\ast}(q)=0$ whenever $q$ is not an integer. We therefore get that
	\[F(\chi,s) = \sum_{n=1}^\infty \widehat{f^\ast}(n) n^{-s} \chi(n)\]
	as the Fourier series of a function in $H^1(\mathbb{T}^\infty)$. 
	
	If $s$ is in $\mathbb{C}_{1/2}$, then it follows from Lemma~\ref{lem:helsonabs} that this Fourier series is absolutely convergent on $\mathbb{T}^\infty$. If $\chi$ is in $E_f$ and $s$ is in $\mathbb{C}_{1/2}$, then we get that $F(s,\chi)=f_\chi(s)$ as absolutely convergent series. This completes the proof of the first assertion, since the Poisson integral $F(\chi,s)$ is well-defined in $\mathbb{C}_0$ and analytic in $\mathbb{C}_{1/2}$, so it must be analytic in $\mathbb{C}_0$.
	
	It remains to establish the estimate \eqref{eq:supsupchi} and in view of what we have done, we write $f_\chi(s)=F(\chi, s)$ for $\chi$ in $E_f$ and $s$ in $\mathbb{C}_0$. Using H\"older's inequality and Tonelli's theorem as before, we obtain
	\[\frac{1}{2T} \int_{-T}^T |f_\chi(\sigma+it)|^p \,dt \leq \int_{-\infty}^\infty |f^\ast(\chi \mathfrak{p}^{-i\tau})|^p \frac{1}{2T} \int_{-T}^T \frac{\sigma}{\sigma^2+(t-\tau)^2}\,dt \,\frac{d\tau}{\pi}.\]
	The estimate
	\[\frac{\mathbf{1}_{[-T,T]}(t)}{2T} \leq \frac{T}{T^2+t^2}\]
	holds for every real number $t$, so we get from the Poisson integral of the Poisson kernel that
	\[\frac{1}{2T} \int_{-T}^T \frac{\sigma}{\sigma^2+(t-\tau)^2}\,dt \leq \int_{-\infty}^\infty \frac{T}{T^2+t^2} \frac{\sigma}{\sigma^2+(t-\tau)^2}\,dt = \pi \frac{T+\sigma}{(T+\sigma)^2 + \tau^2}.\]
	Let $C=C_f(\chi)$ be the constant defined above. Using integration by parts as before, we find in this case that
	\[\frac{1}{2T} \int_{-T}^T |f_\chi(\sigma+it)|^p \,dt \leq 2C \int_{-\infty}^\infty \frac{(1+|\tau|)(T+\sigma)|\tau|}{\left((T+\sigma)^2 + \tau^2\right)^2}\,d\tau = C \left(\frac{2}{T+\sigma} + \pi\right).\]
	Since $T\geq1$ and $\sigma>0$, we obtain the stated result using that $2+\pi \leq 6$. 
\end{proof}

\begin{remark}
	The combination of \eqref{eq:pest1} and \eqref{eq:pest2} from the first part of the proof of Theorem~\ref{thm:ergodicpoisson} supplies the estimate
	\[|f_\chi(s)|^p \leq 2C_f(\chi) \left(\frac{1+|t|}{\pi \sigma} + \frac{1}{2}\right) \leq\sqrt{2}C_f(\chi)\frac{1+|s|}{\sigma}.\]
	The corresponding result for $p=2$ was obtained by Hedenmalm, Lindqvist, and Seip~\cite{HLS1997}*{Theorem~4.2} using the ergodic theorem in a similar manner. The same estimate (with $\sqrt{2}$ replaced by $12$) can also be obtained from the second part of Theorem~\ref{thm:ergodicpoisson} and Lemma~\ref{lem:pest}.
\end{remark}

\begin{proof}
	[Proof of Theorem~\ref{thm:fatou}] Let us define
	\[E = \left\{\chi \in E_f \,:\, \lim_{\sigma \to 0^+} f_\chi(\sigma) = f^\ast(\chi)\right\}.\]
	This set is measurable since $f^\ast$ is measurable and since $\chi \mapsto f_\chi(\sigma)$ is measurable for each fixed $\sigma$. (Recall from the proof of Theorem~\ref{thm:ergodicpoisson} that the function $f_\chi(\sigma) = F(\chi,\sigma)$ is in $H^1(\mathbb{T}^\infty)$.) We will now apply the ergodic theorem for the Kronecker flow to the indicator function $\mathbf{1}_E$, which is plainly integrable. We infer that there is a subset $F$ of $\mathbb{T}^\infty$ of full measure such that if $\chi$ is in $F$, then 
	\begin{equation}\label{eq:minftyE} 
		m_\infty(E) = \lim_{T\to\infty} \frac{1}{2T} \int_{-T}^T \mathbf{1}_E(\chi \mathfrak{p}^{-i\tau})\,d\tau. 
	\end{equation}
	However, if $\chi$ is in $E_f$, then it follows from the first assertion in Theorem~\ref{thm:ergodicpoisson} and the classical Fatou's theorem \eqref{eq:fatouclassical} that
	\[\lim_{\sigma \to 0^+} f_\chi(\sigma \mathfrak{p^{-i\tau}}) = f^\ast(\chi \mathfrak{p}^{-i\tau})\]
	for almost every $\tau$ in $\mathbb{R}$. This means that if $\chi$ is in $E_f$, then the integral on the right-hand side of \eqref{eq:minftyE} equals $2T$ for every $T>0$. Since both $F$ and $E_f$ have full measure, their intersection is nonempty. Choosing any $\chi$ belonging to $F \cap E_f$ in \eqref{eq:minftyE}, we find that $m_\infty(E)=1$. 
\end{proof}

\bibliography{cvlf}

\begin{bibdiv}
\begin{biblist}

\bib{Bayart2002}{article}{
      author={Bayart, Fr\'{e}d\'{e}ric},
       title={Hardy spaces of {D}irichlet series and their composition
  operators},
        date={2002},
     journal={Monatsh. Math.},
      volume={136},
      number={3},
       pages={203\ndash 236},
      review={\MR{1919645}},
}

\bib{Bohr1913}{article}{
      author={Bohr, Harald},
       title={\"{U}ber die {B}edeutung der {P}otenzreihen unendlich vieler
  {V}ariablen in der {T}heorie der {D}irichletschen reihe $\sum
  \frac{a_n}{n^s}$},
        date={1913},
     journal={Nachr. Ges. Wiss. G{\"o}ttingen, Math.--Phys. Kl.},
      volume={1913},
       pages={441\ndash 488},
}

\bib{BP2021}{article}{
      author={Brevig, Ole~Fredrik},
      author={Perfekt, Karl-Mikael},
       title={A mean counting function for {D}irichlet series and compact
  composition operators},
        date={2021},
     journal={Adv. Math.},
      volume={385},
       pages={Paper No. 107775, 48},
      review={\MR{4252762}},
}

\bib{Carlson1922}{article}{
      author={Carlson, F.},
       title={Contributions \`{a} la th\'{e}orie des s\'{e}ries de {D}irichlet,
  {N}ote {I}},
        date={1922},
     journal={Ark. Mat. Astr. Fys.},
      volume={16},
      number={18},
       pages={1\ndash 19},
}

\bib{CG1986}{article}{
      author={Cole, Brian~J.},
      author={Gamelin, T.~W.},
       title={Representing measures and {H}ardy spaces for the infinite
  polydisk algebra},
        date={1986},
     journal={Proc. London Math. Soc. (3)},
      volume={53},
      number={1},
       pages={112\ndash 142},
      review={\MR{842158}},
}

\bib{Garnett2007}{book}{
      author={Garnett, John~B.},
       title={Bounded analytic functions},
     edition={first},
      series={Graduate Texts in Mathematics},
   publisher={Springer, New York},
        date={2007},
      volume={236},
      review={\MR{2261424}},
}

\bib{HR1964}{book}{
      author={Hardy, G.~H.},
      author={Riesz, M.},
       title={The general theory of {D}irichlet's series},
      series={Cambridge Tracts in Mathematics and Mathematical Physics},
   publisher={Stechert-Hafner, Inc., New York},
        date={1964},
      volume={No. 18},
      review={\MR{185094}},
}

\bib{HW1941}{article}{
      author={Hartman, Philip},
      author={Wintner, Aurel},
       title={On the convexity of averages of analytic almost periodic
  functions},
        date={1941},
     journal={Amer. J. Math.},
      volume={63},
       pages={581\ndash 583},
      review={\MR{4337}},
}

\bib{HLS1997}{article}{
      author={Hedenmalm, H{\aa}kan},
      author={Lindqvist, Peter},
      author={Seip, Kristian},
       title={A {H}ilbert space of {D}irichlet series and systems of dilated
  functions in {$L^2(0,1)$}},
        date={1997},
     journal={Duke Math. J.},
      volume={86},
      number={1},
       pages={1\ndash 37},
      review={\MR{1427844}},
}

\bib{Helson1969}{article}{
      author={Helson, Henry},
       title={Compact groups and {D}irichlet series},
        date={1969},
     journal={Ark. Mat.},
      volume={8},
       pages={139\ndash 143},
      review={\MR{285858}},
}

\bib{Helson2006}{article}{
      author={Helson, Henry},
       title={Hankel forms and sums of random variables},
        date={2006},
     journal={Studia Math.},
      volume={176},
      number={1},
       pages={85\ndash 92},
      review={\MR{2263964}},
}

\bib{QQ2020}{book}{
      author={Queff\'{e}lec, Herv\'{e}},
      author={Queff\'{e}lec, Martine},
       title={Diophantine approximation and {D}irichlet series},
     edition={second},
      series={Texts and Readings in Mathematics},
   publisher={Hindustan Book Agency, New Delhi; Springer, Singapore},
        date={2020},
      volume={80},
      review={\MR{4241378}},
}

\bib{SS2009}{article}{
      author={Saksman, Eero},
      author={Seip, Kristian},
       title={Integral means and boundary limits of {D}irichlet series},
        date={2009},
     journal={Bull. Lond. Math. Soc.},
      volume={41},
      number={3},
       pages={411\ndash 422},
      review={\MR{2506825}},
}

\bib{Titchmarsh1958}{book}{
      author={Titchmarsh, E.~C.},
       title={The theory of functions},
   publisher={Oxford University Press, Oxford},
        date={1958},
        note={Reprint of the second (1939) edition},
      review={\MR{3155290}},
}

\end{biblist}
\end{bibdiv}

\end{document}